\documentclass[11pt]{amsart}
\newcommand{\bt}{\begin{Theorem}}
\newcommand{\et}{\end{Theorem}}
\newcommand{\bi}{\begin{itemize}}
\newcommand{\ei}{\end{itemize}}
\newcommand{\bea}{\begin{eqnarray}}
\newcommand{\ba}{\begin{array}}
\newcommand{\eea}{\end{eqnarray}}
\newcommand{\ea}{\end{array}}

\newcommand{\what}{\widehat}
\newcommand{\lgra}{\longrightarrow}

\newtheorem{Definition}{Definition}[section]
\newtheorem{Theorem}[Definition]{Theorem}
\newtheorem{Lemma}[Definition]{Lemma}

\newtheorem{Proposition}[Definition]{Proposition}
\newtheorem{Corollary}[Definition]{Corollary}

\newtheorem{Remark}[Definition]{Remark}

\newcommand{\be}{\begin{equation}}
\newcommand{\ee}{\end{equation}}
\newcommand{\bvb}{\begin{verbatim}}

\newcommand{\evb}{\end{verbatim}}

\newcommand{\wtilde}{\widetilde}%
\newcommand{\R}{\mathbb R}%
\newcommand{\C}{\mathbb C}%

\begin{document}
\baselineskip16pt
\author[Pratoosh Kumar]{Pratyoosh Kumar}
\address[Pratyoosh Kumar]{State-Math Unit, Indian Statistical Institute 203 B.T.Road, Kolkata. 700108, India,
E-mail: pratyoosh(uderscore)v(at)isical(dot)ac(dot)in}

\title[Weak $L^2$ Eigenfunctions of symmetric spaces]
{Fourier restriction Theorem and characterization of weak $L^2$ eigenfunctions of the Laplace--Beltrami operator}
\thanks{This work is supported by vising scientist fellowship of Indian Statistical Institute, Kolkata (India)}

\begin{abstract} In this paper we prove the Fourier restriction theorem for $p=2$ on Riemannian symmetric spaces of noncompact type 
with real rank one which extends the 
earlier result proved in \cite[Theorem 1.1]{KRS}. This result depends on the weak $L^2$ estimates of the Poisson transform of $L^2$
function. By using this estimate of the Poisson transform we also characterizes all weak $L^2$ eigenfunction of the Laplace--Beltrami 
operator of Riemannian symmetric spaces of noncompact type with real rank one and eigenvalue $-(\lambda^2+\rho^2)$ 
for $\lambda\in\R\setminus\{0\}$.  
\end{abstract}
\maketitle

\section{Introduction}
Let $X=G/K$ be a rank one symmetric space of non compact type. For $\lambda \in \mathfrak{a^*_{\C}}$ the Poisson transform
 $\mathcal{P}_\lambda F$ of $F\in L^1(K/M),$ is given by 
$$\mathcal{P}_\lambda F(x)=\int_{K/M} e^{(i\lambda+\rho)A(x,b)}F(b)db.$$
For the meaning of symbols we refer the reader to section 2.
It is well known that $\mathcal{P}_\lambda F$ is an eigenfunction of the Laplace--Beltrami operator $\Delta$
with eigenvalue $-(\lambda^2+\rho^2)$ \cite[page 100]{He5}. A celebrated result of Helgason (and Kashiwara et al.
for higher rank) says that if $u$ is an eigenfunction of Laplace--Beltrami operator on $X$ then $u$  is the 
Poisson transform of an analytic functional $T$ defined on $K/M$ \cite[Chapter V, Theorem 6.6]{He5}.

 In this paper we are interested only in such eigenfunctions which are Poisson transforms of $L^1(K/M)$ functions.
 In \cite{F} Frustenberg proved that $u$ is a bounded harmonic function (i.e. $\Delta u=0$) if and only if $u$ is a Poisson integral 
of bounded function on $K/M$ (the Poisson transform corresponding to $\lambda=-i\rho$ is known as the Poisson integral).
The characterization of eigenfunctions which are Poisson transform of an $L^p$ function was studied extensively from then onwards
 \cite{LR,S,SOS,ACD,Str1,Io3,BS}. 
The following $L^p$ analogue  of  Frustenberg's result was proved in \cite{LR,S} for other values of $p$:

{\it Let $1\leq p<2,$ $\lambda=\alpha+i\gamma_{p'}\rho,$ $\alpha \in \R$ and $\Delta u=-(\lambda^2+\rho^2)u$. Then
$u\in L^{p',\infty}(X)$ if and only if $u=\mathcal{P}_\lambda F$ for some $F\in L^{p'}(K/M)$. Moreover,
\be\|\mathcal{P}_{\alpha+i\gamma_{p'}\rho} F\|_{p',\infty}\leq \|F\|_{p'}.\label{poissonp}\ee
If $p=2$ then $u=\mathcal{P}_{0} F$ with $F\in L^2(K/M),$ if and only if  $\Delta u=-\rho^2u$ and the function 
$h(ka_t.o)=(1+t)^{-1}u(ka_t.o)$ is in $L^{2,\infty}(X)$.}

In the above $\gamma_p=(\frac2p-1),\; 1\leq p\leq \infty,$ and hence $\gamma_{p'}=-\gamma_p$ where 
$\frac1 p+\frac1 {p'}=1$.
The estimate (\ref{poissonp}) played a fundamental role in the development of harmonic analysis on Riemannian symmetric spaces.
 By using the estimate (\ref{poissonp}) for $1<p<2$, 
 Lohou{\'e} and Rychener \cite{LR} obtained the following continuous inclusion 
$$L^{p,1}(G)\ast L^{p,1}(G/K)\subseteq L^{p,1}(G).$$
Cowling et al. in \cite{Co2} proved an elegant generalization of this result. They proved a convolution relation of the following
 form \be L^{p,r}(G)\ast L^{p,s}(G)\subseteq L^{p,t}(G),\label{cowling}\ee
where $1<p<2$, $\frac1r+\frac1s\geq\frac1t$ and $r,s,t \in [1,\infty]$. 
Finally in \cite{Io1} Ionescu obtained the end point version of (\ref{cowling})
\be L^{2,1}(G)\ast L^{2,1}(G)\subseteq L^{2,\infty}(G).\label{endpoint}\ee 

Coming back to estimate \ref{poissonp}, we note that this generalizes the well known behavior of the elementary spherical  
function $\phi_\lambda$ in terms of certain natural size estimates of the Poisson transform. For the reader's 
benefit we will explain this point in detail.
For $\lambda\in\C$ the elementary spherical function $\phi_\lambda$ is given by 
$$\phi_\lambda (x)=\mathcal{P}_\lambda 1(x)=\int_{K/M} e^{(i\lambda+\rho)A(x,b)}db.$$ 
If $1\leq p<2$ and $\alpha\in\R$ then the well known estimate $|\phi_{\alpha+i\gamma_{p'}\rho}(a_t)|\asymp e^{-2\rho t/p'}\;(t\geq0)$ 
(\cite[Lemma 3.2]{RS}) implies that $\phi_{\alpha+i\gamma_{p'}\rho}$ belongs to $L^{p',\infty}(G)$ (the notation $U\asymp W$ for two
 positive functions mean that there exist positive constants $c_1,\;c_2$ 
such that $c_1U(x)\leq W(x)\leq c_2 U(x)$ for appropriate values of $x$). The estimate (\ref{poissonp}) can be 
thought of as a generalization of this fact. However, the case $p=2$ is  little different. 
It is known from the work of Harish-Chandra and a subsequent refinement by Anker that
 $\phi_0(a_t)\asymp (1+t)^{-1}e^{-\rho t}\;(t\geq0)$
(\cite[page 656]{ADY}). Hence $\phi_0\notin L^{2,\infty}(G)$. But it is not hard to show that $\phi_0\in L^q(G)$ for all $q>2$. 
This was generalized in \cite{Co1} in the form of the following estimate: for all $q\in (2,\infty]$ there exists a constant 
$C_q>0$ such that for all $F\in L^2(K/M)$ one has the 
inequality
$\|P_0F\|_{L^q(X)}\leq C_q\|F\|_{L^2(K/M)}$.
We now focus on the case $\lambda\in\R\setminus \{0\}$.
It follows from the Harish-Chandra series for elementary
spherical functions that for $\lambda \in \R\setminus\{0\}$ the function $\phi_{\lambda}$ satisfies the stronger estimate 
$|\phi_\lambda(a_t)|\leq C_\lambda e^{-\rho t}$ for all $t\geq0$ (\cite[(3.11)]{Io2}). Hence $\phi_{\lambda}\in L^{2,\infty}(G)$ for
 $\lambda\in\R\setminus \{0\}$. 
We will show that this also holds true for Poisson transforms $P_{\lambda}F$ for $F\in L^2(K/M)$ and $\lambda \in \R\setminus\{0\}$.
As all radial eigenfunction of $\Delta$ with $-(\lambda^2+\rho^2)$ are constant multiple of 
$\phi_\lambda$ it follows that for a given $\lambda \in \R\setminus\{0\}$ all radial eigenfunctions of $\Delta$ with eigenvalue 
$-(\lambda^2+\rho^2)$  belongs to $L^{2,\infty}(X)$. One of our aim in this paper is to characterize all weak $L^2$ eigenfunctions of 
$\Delta$ with eigenvalue $-(\lambda^2+\rho^2)$ for a given $\lambda \in \R\setminus\{0\}$.  

The main result we prove in this paper is the following:

\begin{Theorem}\label{main}
If $\lambda \in \R\setminus{\{0\}}$ then there exists a constant $C_\lambda>0$ such that  
\be\|\mathcal{P}_\lambda F\|_{2,\infty}\leq C_\lambda \|F\|_{L^2(K/M)},\;\;\;\ \text{for all}\;\ F\in L^2(K/M).\label{estimate}\ee
If $\Delta u=-(\lambda^2+\rho^2)u$ then $u\in L^{2,\infty}(X)$ if and only if 
$u=\mathcal{P}_\lambda F$ for some $F\in L^2(K/M)$. 
\end{Theorem}
It is important to realize that Poisson transforms are really certain matrix coefficient of the class one principal series
representation. For $\lambda\in \mathfrak a^*_{\C}$ the class one principal series representation $\pi_\lambda$ of a semisimple Lie 
group $G$ are realized on $L^2(K/M)$  and are given by 
$$\pi_\lambda(g)f(kM)=e^{-(i\lambda+\rho)H(g^{-1}k)}F(k(g^{-1}k)).$$
It is known that for $\lambda \in \mathfrak a^*$ the representation $\pi_\lambda$ is unitary and irreducible. It is also known that
 $\pi_\lambda$ and $\pi_{-\lambda}$ are unitarily equivalent. Using the description of $\pi_\lambda$ given above  
 the Poisson transform of $F\in L^2(K/M)$ can also be written as
$$\mathcal{P}_\lambda F(g\cdot o)=\left\langle \pi_\lambda(g)1,\bar F \right\rangle.$$
Theorem \ref{main} thus shows that $|\left\langle \pi_\lambda(\cdot)1, F \right\rangle|$ belongs 
to $L^{2,\infty}(G/K)$ for $\lambda \in \mathfrak a^*\setminus\{0\}$ and $F\in L^2(K/M)$. It is obvious from the estimates of $\phi_0$ mentioned earlier
that such a result is not expected if $\lambda=0$.

\noindent {\bf Acknowledgement:} I would like to thank Swagato K. Ray and Rudra P. Sarkar for suggesting this problem to me.
I would also like to thank them for several discussion on the subject. 

\section{Notation and Preliminaries}
In this section we summarize some standard results on noncompact semisimple Lie group and associated 
symmetric space which will be required. Most of our notation are standard and can be found in \cite{GV,H}. Let $G$ be a connected 
noncompact semisimple Lie group with finite center, and $\mathfrak{g}$ be the Lie algebra of $G$. 
Let $\theta$ be a Cartan involution of $\mathfrak{g}$ and $\mathfrak{g}=\mathfrak{k}\oplus\mathfrak{p}$ be the 
associated Cartan decomposition. Let $K=\text{exp}\;\mathfrak{k}$ be a maximal compact subgroup of $G$ and let $X=G/K$ 
be the associated Riemannian symmetric space. If $o=eK$ denotes the identity coset then for $g\in G$ the quantity $r(g)$ denotes 
the Riemannian distance of the coset $g.o$ from the identity coset. Let $\mathfrak{a}$ be a maximal abelian subspace of 
$\mathfrak{p}$, $A=\text{exp}\;\mathfrak{a}$ be the corresponding subgroup of $G,$ and $M$ the centralizer of $A$ in $K$. 

Now onward we will assume that the group $G$
has real rank one that is dim $\mathfrak{a}=1$. In this case it is well known that the set of nonzero roots is either of the 
form $\{-\alpha,\alpha\}$ or $\{-\alpha,-2\alpha,\alpha,2\alpha\}$. Let $\mathfrak{g}_\alpha$ and $\mathfrak{g}_{2\alpha}$
be the root spaces corresponding to the roots $\alpha$ and $2\alpha$ respectively. Let $\mathfrak{n}= 
\mathfrak{g}_\alpha\oplus\mathfrak{g}_{2\alpha}$ and $N=$exp $\mathfrak{n}$. Let $H_0$ be the unique element of $\mathfrak{a}$ 
such that $\alpha(H_0)=1$ and $A=\{a_s : a_s=\text{exp}\;sH_0, s\in\R\}$. We identify 
$\mathfrak{a^*}$ (the dual of $\mathfrak{a}$) and $\mathfrak{a^*_{\C}}$ (the complex dual of $\mathfrak{a}$) by $\R$ and $\C$
via the identification $t\mapsto t\alpha$ and $z\mapsto z\alpha$, $t\in \R$ and $z\in \C$ respectively. Let 
$m_1=\text{dim}\;\mathfrak{g}_\alpha$, $m_2=\text{dim}\;\mathfrak{g}_{2\alpha}$ and $\rho=\frac{1}{2}(m_1+2m_2)\alpha$ be the
half sum of positive roots. By abuse of notation we will denote $\rho(H_0)= \frac{1}{2}(m_1+2m_2)$ by $\rho$.
Let $G=KAN$ be the Iwasawa decomposition of $G$ that is , any $g\in G$ can be uniquely written as 
\be g=k(g)\;\text{exp}H(g)\;n(g)\label{iwasawa}\ee
where $k(g)\in K$, $H(g)\in\mathfrak{a}$ and $n(g)\in N$.
For $F\in L^1(K/M)$ and $\lambda\in \C\;(=\mathfrak{a}^*_{\C})$ the Poisson transform $P_{\lambda}F$ is a function on $X$ defined by 
the formula \be \mathcal{P}_\lambda F(x)=\int_{K/M} e^{(i\lambda+\rho)A(x,b)}F(b)db\label{poissonX},\ee
where $A(gK,kM)=-H(g^{-1}k).$

Let $\bar{N}=\text{exp}(\mathfrak{g}_{-\alpha}\oplus\mathfrak{g}_{-2\alpha})$ and $G=\bar NAK$ be the corresponding Iwasawa 
decomposition of $G$. Let $dk,d\bar n$ and $dm$ be the normalized Haar measure on $K,\bar N$ and $M$ such that 
\bea \int_{K}1\; dk=1;\;\;\ \int_M 1\; dm=1; \;\;\ \int_{\bar N} e^{-2\rho(H(\bar n))}d\bar n=1
\nonumber &&\\ \int_G f(g)\;dg=c\int_{\bar N}\int_{\R}\int_K f(\bar na_tk)e^{2\rho t}\;d\bar n\; dt\; dk.\eea
We will also need the following change of variable formula relating integrals on $K/M$ and integrals on $\bar N$ \cite[Chapter I, Theorem 5.20]{H}
\be\int_{K/M}F(b)\;db = \int_{\bar N} F(k(\bar n)M)e^{-2\rho(H(\bar n))}\;d\bar n.\label{ktonbar}\ee
We also have the cartan decomposition $G=K\overline{A^+}K$ where $\overline {A^+}=\{a_t\in A  : t\geq 0\}$. The functions defined on
 $X$ can also be viewed as right $K$-invariant function on $G$. The $K$-biinvariant functions of $G$ are called radial function. 
The Haar measure related to the Cartan decomposition is given by 
\be \int_G f(g)\; dg =C \int_K\int_0^\infty \int_K f(k_1a_tk_2)(\sinh t)^{m_1}(\sinh 2t)^{m_2}\;dk\;dt\;dk.\label{cartan}\ee 
 The nilpotent subgroup $\bar N$ can be identified with $\R^{m_1}\times\R^{m_2}$ via a
 natural map $\bar n=\exp (V+Z)\rightarrow (V,Z),$
where $m_1=\text{dim}\;\mathfrak{g}_{-\alpha}$, $m_2=\text{dim}\;\mathfrak{g}_{-2\alpha}$, $V\in \R^{m_1}$ and $Z\in \R^{m_2}$.
For any $t\in \R$ the dilation $\delta_t$ on subgroup $\bar N$ is an automorphism given by $\delta_t(\bar n)=a_{-t}\bar na_t$.
Writing $\bar n=(V,Z)$ it can be written more explicitly as: $\delta_t(V,Z)=(e^tV,e^{2t}Z).$ We 
define the function $|\bar n|$ on $\bar N$ by 
\be |\bar n|=|(V,Z)|=(c^2|V|^4+4c|Z|^2)^{1/4},\label{norm}\ee
where $c=\frac {1}{4(m_1+m_2)}$. This function
has the property that $|\delta_t(\bar n)|=e^t|\bar n|$ 
for any $s\in\R$ and $\bar n\in \bar N$. For $\lambda \in \R\setminus\{0\}$ we define the kernel $K_\lambda$ by the formula
\be K_\lambda(\bar n)=|\bar n|^{-(Q+i2\lambda)},\;\; \text{for}\;\ |\bar n| \neq 0 \ee
where $Q=2\rho$ is the homogeneous dimension of $\bar N$. $L^2$ boundedness of the convolution operator defined by kernel $K_\lambda$
have been studied in \cite[Section I]{KS} (see also \cite[Theorem 6.19]{FS} ). We define the truncated kernels $K_{\lambda,\eta}$ by
$$K_{\lambda,\eta}(\bar n)=  K_\lambda(\bar n)\chi_{\{\bar n\in\bar N\mid |\bar n|\geq\eta\}}(\bar n),\;\;\;\;\bar n\in\bar N,$$
and the corresponding convolution operator $T_\eta$ by 
$$T_\eta \psi(\bar n)=\int_{\bar N}\psi(\bar n_1)K_{\lambda,\eta}(\bar n_1^{-1}\bar n)d\bar
 n_1,\;\;\;\; \psi\in C_c^\infty(N), \bar n\in\bar N.$$
We will need to consider the following maximal operator associated to the truncated kernel $K_{\lambda,\eta}$ defined by
$$T_*\psi(\bar n)=\sup_{\eta>0}|T_\eta \psi(\bar n)|=\sup_{\eta>0}\left|\int_{|\bar n_1|\geq \eta}\psi(\bar n\bar n_1)
\frac{1}{|\bar n_1|^{Q+i2\lambda}}\;d\bar n_1\right|.$$ 
By using the argument given in \cite[page 33-page 36]{Stn} (see also \cite[page 627]{Stn})
we have the following result regarding the operator $T_*$
\begin{Theorem}\label{important}
If $\lambda\in\R\setminus \{0\}$ then there exists a constant $C_{\lambda}>0$ such that
\be\|T_*\psi\|_{L^2(\bar N)}\leq C_\lambda \|\psi\|_{L^2(\bar N)},\ee
for all $\psi\in L^2(\bar N)$.
\end{Theorem}
In the following we collect some basic facts about Lorentz spaces which will be used in this paper (see \cite{G, SW} for details). 
Let $(M, m)$ be a $\sigma$-finite measure space, $f:M\lgra \C$ be a measurable function. 
The distribution function $d_f: (0, \infty)\lgra (0,\infty]$ and nonincreasing rearrangement
$f^*: (0, \infty)\lgra (0,\infty]$ of $f$ are defined by formulae
$$ d_f(s)= m(\{x \in M): |f(x)|> s\})\;\;\;\ \text{and}\;\;\ f^*(t)=\inf\{s\mid
d_f(s)\le t\}.$$

For $p\in [1, \infty)$, $q\in [1, \infty]$ we
define,
\begin{equation}\|f\|^*_{p,q}=\begin{cases}\left(\frac qp\int_0^\infty [f^*(t)t^{1/p}]^q\frac{dt}t\right)^{1/q}
\ \textup{ when } q<\infty\\ \\ \sup_{t>0}t^{1/p}f^*(t)=\sup_{t>0}td_f(t)^{1/p}\ \ \ \ \ \ \ \ \textup{ when } q=\infty.
\end{cases}\label{lorentznorm}\end{equation} 
 
For $p\in [1, \infty)$, $q\in [1, \infty]$ we define the {\em Lorentz space} $L^{p,q}(M)$  as follows:
$$L^{p,q}(M) = \{f:M\lgra \C : f \quad \text{measurable and}\quad \|f\|^*_{p,q}<\infty \}.$$
 
By $L^{\infty, \infty}(M)$ and $\|\cdot\|_{\infty, \infty}$ we mean respectively
the space $L^\infty(M)$ and the norm $\|\cdot\|_\infty$ we also have
$L^{p,p}(M)=L^p(M)$. For $1<p<\infty$ the space $L^{p,\infty}$ is known as weak $L^p$ space and also 
$L^{p,q}\subset L^{p,s}$ for all $1\leq q\leq s\leq \infty$. For $1<p<\infty$ and
$1\le q<\infty$, the dual (the space of all continuous linear
functional) of $L^{p,q}(M)$ is $L^{p', q'}(M)$. Everywhere in this paper
any $p\in [1, \infty)$ is related to $p'$ by the relation $\frac 1p+\frac 1{p'}=1.$

We will follow the standard practice of using the letter C for constant, whose value may
change from one line to another. Occasionally the constant C will be suffixed to show its dependency
on related parameter. The letters $\C$ and $\R$ will denote the set of complex and real
numbers respectively.

\section{Proof of Theorem \ref{main}}
The important part of the proof of Theorem \ref{main} is the following norm estimate.
Our proof of the following Proposition uses an argument similar to one given in \cite{So}.
\begin{Proposition}
If $\lambda\in \R \setminus{\{0\}}$ then there exists a constant $C_{\lambda}>0$ such that for all $F\in L^2(K/M)$,
$$\|\mathcal{P}_\lambda F\|_{L^{2,\infty}(X)}\leq C_\lambda \|F\|_{L^2(K/M)}.$$
\label{fundamental}
\end{Proposition}
\begin{proof}
In view of the Iwasawa decomposition $G=\bar{N}AK$ the symmetric space $X$ can be identified with
 $\bar{N}\times\R$ via the map $(\bar{n},t)\rightarrow \bar{n}a_t\cdot o$. If $F\in L^2(K/M)$ and $\bar na_t\cdot o \in X$, 
then from the definition (\ref{poissonX}) of the Poisson transform and (\ref{ktonbar}) we have,
$$\mathcal{P}_\lambda F(\bar n a_t\cdot o)=\int_{\bar N} e^{-(i\lambda+\rho)H(a_{-t}\bar n^{-1}k(\bar m))}F(k(\bar m)M)
e^{-2\rho(H(\bar m))}d\bar m .$$
 As $A$ normalizes $N$ it follows from the Iwasawa decomposition (\ref{iwasawa}) that
$$-H(a_{-t}\bar n^{-1}k(\bar m))=-H(\delta_t(\bar n^{-1}\bar m))+H(\bar m)+t\label{hnbar}$$
(see \cite [page 518]{Io2}). By using the last two equality we get,
\bea\mathcal{P}_\lambda F(\bar n a_t\cdot o)&=& e^{(i\lambda+\rho)t}\int_{\bar N} 
e^{-(i\lambda+\rho)H(\delta_t(\bar n^{-1}\bar m)}F(k(\bar m)M) e^{(i\lambda-\rho)(H(\bar m))}d\bar m\nonumber \\
&=& e^{(i\lambda+\rho)t}\int_{\bar N} 
e^{-(i\lambda+\rho)H(\delta_t(\bar m))}\psi(\bar n \bar m)\; d\bar m,\label{prop1}
\eea
where $\psi$ is a function defined on $\bar N$ given by $\psi(\bar m)=F(k(\bar m)M) e^{(i\lambda-\rho)(H(\bar m))}$. 
Using the integral formula (\ref{ktonbar}) one has $\|\psi\|_{L^2(\bar N)}=\|F\|_{L^2(K/M)}$.
If we write $\bar m=(V,Z)$ then $e^{\rho H(\bar m)}$ is given by the formula 
$$ e^{\rho(H(\bar m))}=[(1+c|V|^2)^2+4c|Z|^2]^{\frac 1 4(m_1+2m_2)},$$
where $c^{-1}=4(m_1+4m_2)$ (see \cite[Chapter II, Theorem 6.1]{He5}). Thus from the above formula, definition of norm 
(\ref{norm}) and relation $\delta_t(\bar m)=(e^tV,e^{2t}Z)$ we have,
\bea e^{-(i\lambda+\rho)H(\delta_t(\bar m))}&=&
\frac{1}{[(1+c|e^tV|^2)^2+4c|e^{2t}Z|^2]^{(i\lambda+\rho)/2}}\nonumber\\
&=& \frac{1}{e^{2(i\lambda+\rho)t}[(e^{-2t}+c|V|^2)^2+4c|Z|^2]^{(i\lambda+\rho)/2}}\nonumber\\
&=& \frac{1}{e^{2(i\lambda+\rho)t}[(e^{-4t}+2ce^{-2t}|V|^2)+|\bar m|^4]^{(i\lambda+\rho)/2}}.\label{prop2}
\eea
By using (\ref{prop2}) in (\ref{prop1})  we get the following expression of the Poisson transform 
\be\mathcal{P}_\lambda F(\bar n a_t\cdot o)= e^{-(i\lambda+\rho)t}\int_{\bar N} 
\frac{1}{[(e^{-4t}+2ce^{-2t}|V|^2)+|\bar m|^4]^{(i\lambda+\rho)/2}} \psi(\bar n\bar m)\;d\bar m. \label{proposition}
\ee
We will now follow an argument of  Sj\"ogren  \cite[page 108]{So} to dominate the above integral by two maximal 
functions whose $L^p$ behaviors are known. 
We first restrict the integral over the ball $ B(e^{-t})=\{\bar m\in\bar N : |\bar m|\leq e^{-t}\}$ in $\bar N$. 
Since the Haar measure of $\bar N$ is the Lebesgue measure it follows that the measure of $B(e^{-t})$ is proportional to $e^{-Qt}$.
 In this case we have,
\bea \left|\int_{B(e^{-t})} 
\frac{1}{[(e^{-4t}+2ce^{-2t}|V|^2)+|\bar m|^4]^{(i\lambda+\rho)/2}} \psi(\bar n\bar m)\;d\bar m\right|
&\leq& \frac{1}{e^{-Qt}} \int_{B(e^{-t})} |\psi(\bar n\bar m)|\;d\bar m\nonumber\\
&\leq& C M_0\psi(\bar n),\label{first}
\eea 
where $M_0$ is the standard Hardy--Littlewood maximal operator on $\bar N$. It is well known that 
the operator $M_0$ is bounded from $L^p(\bar N)$ to $L^p(\bar N)$ for $1<p\leq \infty$ \cite[Theorem 1, page 13]{Stn}.
For $|\bar m|>e^{-t},$ we will compare the kernel with $|\bar m|^{-(Q+i2\lambda)}$. We claim that,
\be \left| \frac{1}{[(e^{-4t}+2ce^{-2t}|V|^2)+|\bar m|^4]^{(i\lambda+\rho)/2}} - \frac{1}{|\bar m|^{(Q+i2\lambda)}}\right|
\leq C \frac{e^{-t}}{|\bar m|^{Q+1}},\label{compare1}\ee
for all $\bar m \in B(e^{-t})$. Consider the function $\phi:(0,\infty)\lgra \C$ defined by $\phi(r)=r^{-Q-2i\lambda}$. If we take 
$r=|\bar m|$ and $s=[(e^{-4t}+2ce^{-2t}|V|^2)+|\bar m|^4]^{1/4}$ then it follows from the mean value theorem that there exists
 $r_0 \in (r,s)$ such that 
$$|\phi(s)-\phi(r)|=|s-r||\phi{'}(r_0)|=C|s-r|r_0^{-Q-1}.$$ 
Since $e^{-t}<|\bar m|$ we have $|\bar m|\leq r_0\leq 3 |\bar m|$ and $|s-r|\leq C e^{-t}$.
 This proves our claim. We also have the following estimate 
\bea \int_{|\bar m|>e^{-t}} \frac{e^{-t}}{|\bar m|^{Q+1}}|\psi(\bar n\bar m)|\; d\bar m
&=& \sum _{j=0}^{\infty}\int_{2^j e^{-t}<|\bar m|\leq 2^{j+1}e^{-t}}\frac{e^{-t}}{|\bar m|^{Q+1}}|\psi(\bar n\bar m)|\; d\bar m
\nonumber\\ &\leq&  \sum _{j=0}^{\infty}\frac{e^{-t}}{(2^j e^{-t})^{Q+1}}\int_{B(2^{j+1}e^{-t})}|\psi(\bar n\bar m)|\; d\bar m
\nonumber \\&\leq& C_Q\sum _{j=0}^{\infty}\frac{1}{2^j|B(2^{j+1}e^{-t})|}\int_{B(2^{j+1}e^{-t})}|\psi(\bar n\bar m)|\; d\bar m
\nonumber \\
&\leq&  C_Q M_0\psi(\bar n),\label{maximal} \eea
where $|B(2^{j+1}e^{-t})|$ denotes the Haar measure of $B(2^{j+1}e^{-t})$.
If we combine the inequalities (\ref{compare1}) and (\ref{maximal}) we get 
\be \left|\int_{|\bar m|>e^{-t}} 
\frac{1}{[(e^{-4t}+2ce^{-2t}|V|^2)+|\bar m|^4]^{(i\lambda+\rho)/2}} \psi(\bar n\bar m)\;d\bar m\right|\leq C T\psi
(\bar n),\label{final}\ee
where $T\psi(\bar n)= M_0\psi(\bar n)+T_*\psi(\bar n)$.
From (\ref{proposition}), (\ref{first}) and (\ref{final}) we now have the following pointwise estimate of the Poisson transform 
$$|\mathcal{P}_\lambda F(\bar n a_t\cdot o)|\leq C e^{-\rho t} T\psi(\bar n)).$$
We will use the above estimate to show that the Poisson transform belongs to weak $L^2$. 
In this regard, for any $\beta>0$ we have,

\be
\mu\{\bar na_t \in X | |\mathcal{P}_\lambda F(\bar n a_t\cdot o)|>\beta\} \leq 
\int_{\bar N}\int_{-\infty}^{t_0}e^{2\rho t}dt\; d\bar n= \frac C {\beta^2} \int_{\bar N} T\psi(\bar n)^2\;d\bar n,
\label{prop3}\ee
where $\frac {C T\psi(\bar n)}{\beta^2}=e^{2\rho t_0}$ and $\mu$ denotes the $G$ invariant measure on $X$.
We know from Theorem \ref{important} that the operator  $T_*$ is bounded from $L^2(\bar N)$ to $L^2(\bar N)$ and hence so is 
$T$. 
Therefore from (\ref{lorentznorm}) and (\ref{prop3}) we get
$$\|\mathcal{P}_\lambda F\|_{L^{2,\infty}(X)}\leq C_\lambda \|\psi\|_{L^2(\bar N)} = C_\lambda \|F\|_{L^2(K/M)}.$$
This complete the proof.
\end{proof} 
To complete the proof of the Theorem \ref{main} we will now invoke a result of Ionescu \cite[Theorem 1]{Io2}.
For a locally integrable function $u$ on $X$ define 
$$M(u)=\left(\limsup_ {R\rightarrow \infty}\frac 1 R\int_{(B(o,R))}|u(x)|^2\;dx\right)^{1/2}.$$
It was proved in \cite{Io2} that if $\Delta u=-(\lambda^2+\rho^2)u$ for some $\lambda \in \R\setminus\{0\}$ and $M(u)<\infty$
then $u=\mathcal P_\lambda F$ for some $F\in L^2(K/M)$ (see also \cite{BS}).
Theorem \ref{main} will follow immediately once we prove the following simple lemma (although it was proved in \cite{KRS1}
but for the sake of completeness we provide the detail here).
\begin{Lemma}
 If $\Delta u=-(\lambda^2+\rho^2)u$ for $\lambda\in \R\setminus\{0\}$ and $u\in L^{2,\infty}(X)$
then $$M(u)\leq C \|u\|_{L^{2,\infty}(X)}.$$
\end{Lemma}
\begin{proof}It suffices to show that 
\be\int_{(B(o,R))}|u(x)|^2\;dx\leq C R\|u\|_{L^{2,\infty}(X)}^2 \;\;\ \text{for all}\;\ R>0 \label{enough}.\ee
Since the inequality (\ref{enough}) follows easily for $R\leq 1$ we will concentrate only in the case $R> 1$.
If $u^*$ denotes the decreasing rearrangement of $u$ then it follows from definition of weak $L^2$ spaces that
 \be u^{*}(s)^2\leq\frac{1}{s}\|u\|_{2,\infty}^{2},\quad\mbox{for all $s>0$.}\label{weakl2}\ee
 It follows from the similar argument given in \cite[Lemma 2]{LR} (see also \cite{KRS1}) that 
$$|u(ka_t.o)|\leq C \|u\|_{L^{2,\infty}(X)} e^{\alpha t} \;\;\;\ \text{for some $\alpha >0$ and for all $t>0$}.$$
Since
$(\chi_{B(o,R)})^{*}= \chi_{(0, |B(o,R)|)}$  (see \cite[page 46]{G}) it follows from the above inequality that 
\be (\chi_{B(o,R)}|u|)^{*}(t)^2\leq C\|u\|_{2,\infty}^2 e^{2\alpha R}\chi_{(0, |B(o,R)|)}(t).\label{rearrange-est} \ee
Using the fact that the $G$-invariant measure $\mu B(o,R)$ of the ball $\mu B(o,R)$ 
is propositional to $e^{2\rho R}$ for $R\geq 1$ it follows from 
(\ref{weakl2}) and (\ref{rearrange-est}) that 
\bea\int_{B(o,R)}|u(x)|^2dx&\leq& C \int _0^{\mu B(o,R)} u^*(s)^2 ds\nonumber\\
&\leq& C\int_0^{\mu B(o,R)}\min\left\{\|u\|_{2,\infty}^{2}e^{2\alpha
R},\|u\|_{2,\infty}^{2}\frac{1}{t}\right\}dt\nonumber\\
&\leq& C\|u\|_{2,\infty}^{2}\int_0^{e^{2\rho
R}}\min\left\{e^{2\alpha
R},\frac{1}{t}\right\}dt\nonumber\\
&=& C\|u\|_{2,\infty}^{2}\left(\int_0^{e^{-2\alpha R}}e^{2\alpha
R}dt+\int_{e^{-2\alpha R}}^{e^{2\rho
R}}\frac{dt}{t}\right)\nonumber\\
&=& C\|u\|_{2,\infty}^{2}\left(1+2\rho R+2\alpha
R\right)\nonumber\\
&\leq & C\|u\|_{2,\infty}^{2}R,\label{Rbigger1} \eea
This completes the proof.
\end{proof}
 
An immediate consequence of the above results is  \cite[Proposition 4]{Io2}:
\begin{Corollary}
 If $\lambda \in \mathfrak a^*\setminus \{0\}$ and $F\in L^2(K/M)$ then 
$$\left( \sup_R \frac 1 R \int_{B(o,R)}|\mathcal P_\lambda F(x)|^2\;dx\right)^{1/2}\leq C_\lambda \|F\|_{L^(K/M)}.$$
\end{Corollary}

\section{some consequences}
Theorem \ref{main} has certain consequences which are worth mentioning. One of them is related to the restriction (like) 
theorem for the Helgason Fourier 
transform on $X$. The idea of Fourier restriction theorem on $\R^n (n\geq2)$ originated in the work of Stein. 
The celebrated Tomas--Stein restriction Theorem says that the Fourier transform $\hat{f}$ of a function $f\in L^p(\R^n)$ has 
a well defined restriction on the sphere $S^{n-1}$ 
via the inequality, 
$$\|\what f|_{S^{n-1}}\|_{L^2(S^{n-1})}\leq C_p\|f\|_{L^p(\R^n)}\;\;\;\text{for all}\;\ 1\leq p\leq\frac{2n+3}{n+3}$$    
(see \cite[page 365]{Stn}). We are interested in similar inequalities for the Helgason Fourier transform of suitable functions on $X$. 
Given $f\in C_c^\infty(X)$ the Helgason Fourier transform $\widetilde{f}$ of $f$ is defined by \cite[page 199]{He5} 
\be\widetilde{f}(\lambda,b)=\int_X f(x)e^{(i\lambda+\rho)A(x,b)}\;dx,\;\;\;\ \lambda\in \C,\;\;b\in K/M.\label{fourierx}\ee
For fixed $\lambda \in \C$ the norm inequality of the form
$$\left(\int_{K/M}|\widetilde{f}(\lambda,b)|^qdb\right)^{1/q}\leq C\|f\|_{L^p(X)}$$
can be thought as an analogue of Fourier restriction theorem in the context of a symmetric space $X$. For $1\leq p<2$ the analogue 
of restriction theorem for $L^p$ functions on $X$ with rank~$X=1$ was proved in \cite{LR}. This result was extended for more general
 spaces in \cite{RS}, \cite{KRS}. It was also shown in 
\cite{KRS} that the best posible analogue of restriction theorem for $\lambda\in\R$ is the following: Given $\lambda\in\R\setminus 
\{0\}$ there exists a constant $C_{\lambda,p}>0$ such 
that the following inequality holds $$\left(\int_{K/M}|\widetilde{f}(\lambda,b)|^2db\right)^{1/2}\leq C_{\lambda ,p}\|f\|_{L^p(X)},
\:\:\:\:1\leq p<2 .$$ 
It is the end point case of the above inequality which we are interested in.   
It turns out that this problem can be solved very easily by using estimates of the Poisson transform. 
We first observe that for a given $\lambda\in \C$ and $F\in C^\infty(K/M)$ the Poisson transform $\mathcal{P}_\lambda F$
is related to the Helgason Fourier transform $\widetilde{f}(\lambda, \cdot)$  
as follows: $$\int_{K/M}\widetilde{f}(\lambda,b)F(b)\;db= \int_X f(x)\mathcal{P}_\lambda F(x)\;dx.$$
By using the estimate (\ref{estimate}) and the above equation we have the following version of the restriction Theorem :
\begin{Theorem}
 If $\lambda \in \R\setminus\{0\}$ then there exists a constant $C_\lambda>0$ such that,
\be\left(\int_{K/M}|\wtilde{f}(\lambda, b)|^2 db\right)^{1/2}\le C_\lambda \|f\|_{L^{2, 1}(X)},\;\;\ \text{for all}\;\;f\in L^{2,1}(X).
\label{restX}\ee
\label{rest}\end{Theorem}
\begin{proof}
 If $F\in L^2(K/M)$ then using the fact that the dual of $L^{2,1}(X)$ is $L^{2,\infty}(X)$ and the estimate (\ref{estimate}) we get
\bea\left|\int_{K/M}|\wtilde{f}(\lambda, b)F(b)db\right|&=&\left|\int_X f(x)\mathcal{P}_\lambda F(x)\;dx\right|\nonumber\\
&\leq& \|f\|_{L^{2,1}(X)}\|P_{\lambda}F\|_{L^{2,\infty}(X)}\nonumber\\
&\leq& C_{\lambda}\|f\|_{L^{2,1}(X)}\|F\|_{L^2(K/M)}.\nonumber\eea
\end{proof}
Theorem \ref{rest} can be used to deduce an interesting analytic property of the spectral projection operator 
considered in \cite{Str1}. Authors of \cite{CGM}
used the Kunze Stein phenomenon to prove that the spectral projection operator $f\mapsto f\ast\phi_{\lambda}$ satisfies the estimate 
$\|f\ast\phi_{\lambda}\|_{L^{p'}(X)}\leq C_p\|f\|_{L^p(X)}$ for $\lambda\in\R$ and $1\leq p<2$ (the result is
 valid even if rank~$X>1$).
We now present the following end point estimate of the spectral projection operator which is closely related to
 the behavior of the noncentral 
Hardy Littlewood maximal operator on X (see \cite{Io1}).  
\begin{Corollary}
 If $\lambda \in \R\setminus\{0\}$ then the linear map $f\lgra f\ast \phi_\lambda$ is restricted weak type $(2,2)$.
\end{Corollary}
\begin{proof}
 From \cite[Lemma 4.4]{H} we have the relation 
 $f\ast\phi_\lambda(x)= \mathcal P_{\lambda}(\wtilde f(-\lambda, \cdot))(x)$.
Using (\ref{estimate}), (\ref{restX}) and the above relation we have
\be \|f\ast\phi_\lambda\|_{L^{2,\infty}(X)}= \|\mathcal P_{\lambda}(\wtilde f(-\lambda, \cdot))\|_{L^{2,\infty}(X)}\leq
C_\lambda\|\wtilde f(-\lambda, \cdot)\|_{L^2(K/M)}\leq C_\lambda\|f\|_{L^{2,1}(X)}.\label{corna}\ee
\end{proof}
\begin{Remark}
{\em If $1\leq p<2$ and $\lambda=\alpha+i\gamma_p\rho$ ($\alpha\in\R$) then it follows from 
\cite[Theorem 1.1]{KRS} that $\|f\ast\phi_\lambda\|_{L^{p',\infty}(S)}\leq C_{\lambda}\|f\|_{L^{p,1}(S)}$. 
By using the standard estimate $|\phi_{\alpha+i\gamma_{p'}\rho}(x)|\asymp \kappa_p(x)$ where $\kappa_p$ is the radial 
function defined by $\kappa_p(x)=e^{\frac {-2\rho r(x)}{p'}}$
we get the estimate
\be\|f\ast\kappa_p\|_{L^{p',\infty}(X)}\leq C_{\lambda}\|f\|_{L^{p,1}(X)}\;\;\;  1\leq p<2.\ee
For $p=2$  and $\alpha\in\R\setminus\{0\}$ we have the estimate $|\phi_{\alpha}(x)|\leq C_{\alpha}\kappa_2(x)$. 
Since the spectral projection operator, for 
$\lambda\in\R\setminus\{0\}$, is restricted weak type $(2,2)$ one may ask whether the same holds for the operator 
obtained by convolving with the larger kernel $\kappa_2$.   
By modifying an example given in \cite{Io3} we will show that this is not the case. We first recall the following properties of the 
function $r(x)$ from \cite[page 167]{GV} which will be needed
\be r(ka_t)=r(a_t)=t,\;\;\; r(x)=r(x^{-1}), \;\;\: 
r(xy)\leq r(x)+r(y)\label{remark1}\ee
for all $k\in K$, $x,y\in G$ and $t>0$. We also use the functional equation for the elementary spherical functions 
\cite[Chapter IV]{H}
\be\int_K\phi_0(xky)\;dk=\phi_0(x)\phi_0(y)\label{functional},\ee
and the estimate \cite[page 648]{ADY} 
\be\phi_0(x)\asymp (1+r(x))e^{-\rho r(x)}.\label{phi0}\ee
Consider the radial function $f:X\lgra (0,\infty)$ given by
$f(x)=e^{-\rho r(x)}(1+r(x))^{-3/2}$.
It follows from the calculation \cite[page 92]{Io3} that  
$f\in L^{2,1}(X)$. We will show that $f\ast\kappa_2\notin L^{2,\infty}(X)$. 
By using the Cartan decomposition (\ref{cartan}) we have,
\be f\ast\kappa_2(a_s)=\int_0^\infty e^{-\rho t}(1+t)^{-3/2}\int_K e^{-\rho r(a_{-t}ka_s)}dkJ(t)dt \label{remark2}\ee
where $J(t)=(\sinh t)^{m_1}(\sinh 2t)^{m_2}$.
From (\ref{remark1}), (\ref{functional}) and (\ref{phi0}) and  we get 
$$\int_K e^{-\rho r(a_{-t}ka_s)}dk =\int_K e^{-\rho r(a_{-t}ka_s)}\frac{(1+r( a_{-t}ka_s))}{(1+r( a_{-t}ka_s))}dk
\geq C\frac{1}{1+t+s}\phi_0(a_t)\phi_0(a_s)$$
If we use estimates above, (\ref{phi0}) and $J(t)\asymp e^{2\rho t}$ (for $t\geq 1$) then from (\ref{remark2}) we have
\be f\ast\kappa_2(a_s)\geq C \frac{1+s}{1+2s}e^{-\rho s} \int_1^{s}e^{-2\rho t}(1+t)^{-1/2}e^{2\rho t}dt
\geq C s^{1/2} e^{-\rho s}\ee
for large $s$.
It is now easy to see that $f\ast\kappa_2\notin L^{2,\infty}(X)$.}
\end{Remark}
\noindent
 {\bf Comment regarding the Fourier transform on Damek--Ricci Spaces:}
The notion of Helgason Fourier transform is also meaningful for functions on Damek--Ricci space
\cite{ACD}. It is well known that Damek--Ricci spaces include all  Riemannian symmetric spaces of
 noncompact type with real rank one \cite{ADY}.
The analogue of restriction theorem in this setup was proved in \cite{RS} and \cite{KRS}. 
By using arguments similar to symmetric spaces one can prove the following version of Theorem \ref{rest} for 
Damek--Ricci spaces $S$. 
\begin{Theorem}
 If $\lambda \in \R\setminus\{0\}$ and $f\in L^{2,1}(S)$ then,
\be\left(\int_N|\wtilde{f}(\lambda, n)|^2 dn\right)^{1/2}\le C_\lambda \|f\|_{L^{2, 1}(S)}.\label{restS}\ee
\label{restna}\end{Theorem}
The above theorem extends the following result proved in \cite{KRS}.
 \begin{Theorem}
 Let $f$ be a measurable
function on $S$ and $\alpha\in \R$.
\begin{enumerate}
\item [i)] For $f\in L^{p,1}(S), 1\le p<2$ and $p\le q\le p'$,
\begin{equation*}\left(\int_N|\wtilde{f}(\alpha+i\gamma_{q}\rho,
n)|^q dn\right)^{1/q}\le C_{p,q}\|f\|_{p, 1}, \ C_{1, q}=1.
\end{equation*}

\item [ii)] For $f\in L^{p, \infty}(S), 1<p<2$, $p<q<p'$,
\begin{equation*}\left(\int_N|\wtilde{f}(\alpha+i\gamma_{q}\rho, n)|^q dn\right)^{1/q}\le C_{p,q}\|f\|_{p, \infty}.\end{equation*}
\end{enumerate}
The constants $C_{p,q}>0$ are independent of $\alpha$ and $f$.
Estimates {\em i)} and {\em ii)} are sharp.
\label{restriction-thm-1}
\end{Theorem}

\end{document}